\documentclass[12pt]{amsart}
\usepackage{a4wide}
\usepackage{amsmath,amssymb,amsthm,graphicx}

\usepackage{times}
\usepackage{latexsym}
\usepackage{amscd}
\usepackage{amsfonts}
\usepackage[usenames,dvipsnames]{color}
\usepackage{verbatim}
\usepackage{float}

\newtheorem{prop}{Proposition}[section]
\newtheorem{theorem}[prop]{Theorem}

\newtheorem{lem}[prop]{Lemma}

\theoremstyle{definition}

\newtheorem{definition}[prop]{Definition}

\newtheorem{example}[prop]{Example}

\newtheorem*{rem*}{Remark}
\newtheorem*{notation*}{Notation}
\newtheorem*{example*}{Example}
\newtheorem*{definition*}{Definition}

\newcommand{\bR}{\mathbb{R}}

\newcommand{\cB}{\mathcal{B}}
\newcommand{\cF}{\mathcal{F}}

\newcommand{\End}{\mathrm{End}}
\newcommand{\Pf}{\mathrm{P\! f}}
\newcommand{\vol}{\mathrm{vol}}

\newcommand{\cz}{\mathbb{Z}}

\newcommand{\verti}{\mathrm{vert}}
\newcommand{\hor}{\mathrm{hor}}

\newcommand{\codim}{\mathrm{codim}}

\newcommand{\out}{\mathrm{out}}

\newcommand{\SE}{{S}\! E}
\newcommand{\SM}{{S}\! M}
\newcommand{\cT}{\mathcal{T}}
\newcommand{\Rz}{\mathbb{R}}
\newcommand{\kk}{\mathfrak{k}}
\newcommand{\SVs}{\mathrm{S}V^*}

\newcommand{\Sk}{\mathrm{Sk}}
\newcommand{\cD}{\mathcal{D}}
\newcommand{\cP}{\mathcal{P}}
\newcommand{\Con}{\mathrm{Con}}
\newcommand{\rk}{\mathrm{rk}}

\begin{document}

\title{Higher transgressions of the Pfaffian}
\author{Sergiu Moroianu}
\address{Institutul de Matematic\u{a} al Academiei Rom\^{a}ne\\
P.O. Box 1-764\\
014700 Bucharest\\
Romania}
\email{moroianu@alum.mit.edu}
\date{\today}

\begin{abstract} 
We define transgressions of arbitrary order, with respect to families 
of unit-vector fields indexed by a polytope, for the Pfaffian of
metric connections for semi-Riemannian metrics on vector bundles. We apply this
formula to compute the Euler characteristic of a Riemannian polyhedral manifold
in the spirit
of Chern's differential-geometric proof of the generalized Gauss-Bonnet formula
on closed manifolds 
and on manifolds-with-boundary. As a consequence, we derive an identity for
spherical and 
hyperbolic polyhedra linking the volumes of faces of even
codimension and the measures of outer angles.
\end{abstract}
\maketitle

\section{Introduction}
The classical Gauss-Bonnet theorem computes the Euler characteristic of a closed
surface $(M,g)$ by integrating on $M$ the Gaussian curvature $\mathfrak{k}_g$.
When $M$
has a smooth boundary, one must add a correction term involving the average of
the geodesic curvature of the boundary $\partial M\hookrightarrow M$. If $M$ has
\emph{corners}, i.e., 
the boundary itself has isolated singular points, then the exterior angle of
these corners must also be taken into account. This general formula reads
\[
2\pi \chi(M) = \int_M \mathfrak{k}_g\nu_g - \int_{\partial M} a \cdot l_g
+\sum_{p}\angle^\out(p)
\]
where $a:\partial M\to \Rz$ is the
geodesic curvature function with respect to the outer normal, 
$l_g$ is the induced volume form on $\partial M$, and $\angle^\out(p)$ is the
outer angle at a corner $p$.
This outstanding formula has been generalized to arbitrary dimensions 
by a sequence of authors including H.\ Hopf, H.\ Weyl, C.\ Allendoerfer, 
A.\ Weil, and S.S.\ Chern. 

In the present paper we extend the
Gauss-Bonnet formula of Allendoerfer and Weil \cite{AW} to general compact Riemannian 
polyhedral manifolds (theorems \ref{awt} and \ref{awo}). 

The outline of the paper is as follows. In the first two sections we 
review the Pffafian of the curvature
using the language of double forms. In Section \ref{Spm} we 
define smooth polyhedral manifolds and
polyhedral complexes, and study their properties with respect to 
integration of forms. The category
of polyhedral complexes allows one to bundle together the outer cones of faces of Riemannian 
polyhedral manifolds, the natural locally trivial bundles of spherical polytopes 
where the contributions of the faces are localized. 
Starting from Chern's construction \cite{Ch1} of a transgression form, 
we introduce in Section \ref{sec2} 
higher transgressions for the Pfaffian form on vector bundles endowed 
with a nondegenerate bilinear form and a compatible connection. We show that
the exterior differential of these transgressions can be computed as a
sum of lower-order transgressions.
We then apply our abstract transgression theorem in Section \ref{sec3} to the
case of Riemannian polyhedral manifolds. This formula has been proved with
entirely different methods by  Allendoerfer and Weil \cite{AW} for a particular 
class of polyhedral manifolds that we call \emph{regular}. For regular polyhedral manifolds,
the Gauss-Bonnet formula in the even-dimensional case
follows by iterating the transgression formula on the boundary strata.
In the general case, we use a global polyhedral complex to transfer the transgressions
onto the outer cones via the polyhedral Stokes formula. The 
odd-dimensional case is reduced to the even case by analyzing
the Riemannian product with an interval. 

In the final section, we particularize our formula to space forms:
\begin{theorem}\label{th1}
Let $M$ be a $d$-dimensional compact polyhedral manifold of constant sectional
curvature
$\kk$, with totally geodesic faces. Then
\begin{equation}\label{gbcc}
\frac{\chi(M)}{2}=\sum_{j\geq
0}\sum_{Y\in\cF^{(d-2j)}}\kk^{j}\frac{\vol_{2j}(Y)}{\vol(S^{2j})}\frac{\angle^\out
Y}{\vol(S^{d-2j-1})}
\end{equation}
where $\cF^{(d-2j)}$ is the set of faces of $M$ of dimension $2j$, $S^k$ is the
standard unit sphere in $\bR^{k+1}$, 
and $\angle^\out Y$ is the measure of the outer solid angle at the face $Y$. 
\end{theorem}
By convention, for $2j=d$, both the volume of $S^{-1}$ and the exterior angle of
the interior face of $M$ are defined to be $1$.
We deduce from this theorem identities for hyperbolic polyhedra
involving the volumes of even-dimensional faces and their outer angles,
including an extension to the noncompact case where some - or all - vertices are ideal.

\subsubsection*{Historical note} For submanifolds in $\Rz^n$, the Gauss-Bonnet
formula was stated and proved by Hopf \cite{Hop} for hypersurfaces,
and by Allendoerfer \cite{Alle} for submanifolds of arbitrary codimension. 
Allendoerfer and Weil \cite{AW} derived their formula on Riemannian polyhedra 
mainly as a tool for deducing the Gauss-Bonnet on closed Riemannian manifolds 
without assuming the existence of an isometric embedding in Euclidean space 
(which we now know to exist by Nash's embedding theorem \cite{nas}, 
but was unknown at that time). Their proof is indirect, based on a series of
results: a triangulation theorem for polyhedral manifolds, an additivity result
for the geometric side of the formula, a proof for simplices
embedded in some Euclidean space using Weyl's tube formula, an embedding theorem
for analytic simplices, predating Nash, due to Cartan, and the Whitney analytic
approximation result.

In a series of papers \cite{Ch1}, \cite{Ch2}, Chern gave an entirely different
proof of the Gauss-Bonnet formula, based on his transgression form for the
Pfaffian lifted to the sphere bundle. His construction yields simultaneously 
the necessary correction term on manifolds with boundary, i.e, manifolds with 
corners of codimension $1$. 

We extend here Chern's method to transgressions of higher orders (transgressions
of transgressions), and use these transgressions to prove the Allendoerfer-Weil 
formula for a more general class of polyhedral manifolds than that of \cite{AW}. 
In particular, the additivity of the geometric term in the generalized Gauss-Bonnet 
becomes a corollary of our proof. We close in this way a circle of ideas going back 
to Gauss and Bonnet, Hopf and Weyl, Allendoerfer and Weil and, last but not least, 
S.S.\ Chern. This paper is a tribute to those great mathematicians from the past.

\section{Tensor calculus and the Pfaffian}

We fix in this section the notation concerning vector bundles with metric
connections, and we develop a formalism for multiplying vector-valued forms. 
Such a formalism was already used by Walter \cite{wal} in his generalized
Allendoerfer-Weil formula for locally convex subsets in a Riemannian manifold,
and also by Albin \cite{Alb}.

\subsection{The Pfaffian}

Let $(V,h)$ be a $2n$-dimensional real vector space endowed with a nondegenerate
symmetric bilinear pairing 
of signature $(k,2n-k)$. This means that we can find orthogonal bases
$\{e_1,\ldots,e_{2n}\}$ with
\begin{align*}
h(e_1,e_1)=\ldots=h(e_k,e_k)=1,&& h(e_{k_1},e_{k+1})=\ldots=h(e_{2n},e_{2n})=-1.
\end{align*}
If moreover $V$ is oriented, the \emph{volume form} defined by $h$ is the unique
$2n$-form
$\vol_h\in\Lambda^{2n}V^*$ which takes the value
$1$ on any such basis. This form defines an isomorphism
\begin{align*}
\bR\to\Lambda^{2n} V,&& 1\mapsto \vol_h.
\end{align*}
The inverse $\cB_h:\Lambda^{2n} V\to \bR$ of this isomorphism is called the
\emph{Berezin integral}
with respect to $h$.

Any skew-symmetric endomorphism $A\in\End^-(V)$ determines a $2$-form
$\omega_A\in\Lambda^2V^*$
by
\begin{equation*}
\omega_A(u,v)=h(u,Av).
\end{equation*}
The $n^{\mathrm{th}}$ power of the
$2$-form $\omega_A$ is a multiple of 
$\vol_h$. Define the \emph{Pfaffian} of $A$ with respect to $h$ by
\[
\Pf(A)= \tfrac{1}{n!} \cB_h\left[(\omega_A)^n\right].
\]

\begin{example}
Let $V=\bR^{2n}$ with its euclidean metric and let $\{e_1,\ldots,e_{2n}\}$ be
the standard basis. 
The Pfaffian is a polynomial with integer coefficients (in fact, $\pm 1$) in the
$n(2n-1)$ independent entries of $A$.
It is well-known, and easy to prove, that $\Pf(A)^2=\det(A)$.
\begin{enumerate}
\item For $n=1$, every skew-symmetric matrix takes the form $A=\begin{bmatrix}
0&a\\-a&0\end{bmatrix}$. Then 
\begin{align*}
\omega_A=a e^1\wedge e^2,&&\Pf(A)=a.
\end{align*}
\item For $n=2$ and
$A=\begin{bmatrix}0&a&b&c\\-a&0&d&f\\-b&-d &0 &g\\
-c&-f & -g&0 \end{bmatrix}$,
\begin{align*}\omega_A={}& ae^1\wedge e^2 +be^1\wedge e^3+ce^1\wedge e^4 + d
e^2\wedge e^3
+ f e^2\wedge e^4 +ge^3\wedge e^4,\\ 
\Pf(A)={}& ag - bf+ cd.
\end{align*}
\end{enumerate}
\end{example}

\subsection{Vector-valued forms}
Let $M$ be a smooth manifold, and $E_1, E_2, E_3$ real vector bundles over
$M$. To every 
linear map $p:E_1\otimes E_2\to E_3$ we associate a ``product'' on spaces of
vector-valued forms, 
i.e., a bilinear map
\begin{align*}
P:\Omega^*(M,E_1)\times \Omega^*(M,E_2) \to{}& \Omega^*(M,E_3),\\ 
(\alpha_1\otimes
s_1)\times(\alpha_2\otimes s_2)\mapsto{}& \alpha_1\wedge\alpha_2
\otimes p(s_1,s_2).
\end{align*}
\begin{itemize}
\item A first example of such a product arises for
$E_1= \bR$, $E_2=E_3=E$ and $p:\bR\otimes E\to E$ the canonical isomorphism. We
recover the $\Omega^*(M)$-module structure of $\Omega^*(M,E)$. 
\item
When $E_1=\End(E)$, $E_2=E_3=E$, and $p$ is the tautological pairing
$\End(E)\times E\to E$,
we recover the action of endomorphism-valued forms on $E$-valued forms.
\item
A large class of examples arises when
$E_1=E_2=E_3=E$ are bundles of $\bR$-algebras, and $p$ is the algebra product of
$E$. 
\item As a particular case of the previous example, let $E$ be the bundle of
exterior algebras of a vector bundle $V$. Set
\[\Omega^{u,v}(M,V):=\Omega^u(M,\Lambda^v V^*).\]
When $M$ and $V$ are clear from the context we will suppress them, using
simply the notation $\Omega^{u,v}$. We get a bi-graded algebra structure
on the space of \emph{double forms}
$\Omega^*(M,\Lambda^*V)=\oplus_{u,v}\Omega^{u,v}$.
\item Another particular case: take $E$ to be the endomorphism bundle of some
vector bundle $V$. 
We get a composition product on the space of \emph{endomorphism valued forms}
$\Omega^*(M,\End(V))$.
\end{itemize}

The two last two examples may lead to confusion when $V$ is additionally
endowed with a nondegenerate
symmetric bilinear pairing $h$. In that case, there is 
an identification of the space of $h$-antisymmetric endomorphisms $\End^-(V)$
with $\Lambda^2(V^*)$, given by
\begin{align}\label{omegaa}
\End^-(V)\ni A\mapsto \omega_A\in \Lambda^2(V^*),&& \omega_A(u,v)=h(u,Av).
\end{align}
So there exist two different ``product'' maps on $\Omega^*(M,\Lambda^2(V^*))$:
one is the ``exterior product" taking values in 
$\Omega^*(M,\Lambda^4(V^*))$, the other one is the ``composition product''
obtained by identifying $\Omega^*(M,\Lambda^2(V^*))$ with
$\Omega^*(M,\End^-(V))$ and then using the product of endomorphisms, hence
taking values in $\Omega^*(M,\End(V))$. 

For simplicity, in the sequel we shall write $\alpha \beta$ for the ``product''
$P(\alpha,\beta)\in\Omega^*(M,E_3)$.

\section{The curvature tensor as a double form}
For every vector bundle $E\to M$ with connection $\nabla$ and for every $k\geq
0$ we denote by $d^\nabla:\Omega^k(M,E)\to\Omega^{k+1}$ the
exterior differential twisted by $\nabla$ on $E$-valued forms. Recall that on a
tensor
product $S=\alpha\otimes s$, where $\alpha\in\Omega^k$ and $s\in\Gamma(E)$ are a
locally-defined $k$-form, respectively a section in $E$, $d^\nabla$ is given by
\[d^\nabla (S)= d\alpha\otimes s +(-1)^k \alpha \nabla s.
\]
The first-order differential operator $d^\nabla$ 
is a derivation on the graded $\Omega^*(M)$-module
$\Omega^*(M,E)$, in the sense that for any $\beta\in\Omega^s(M)$ and $S\in
\Omega^*(M,E)$, the Leibniz rule holds:
\[d^\nabla (\beta S)= (d\beta) S+(-1)^s \beta d^\nabla S.
\]

The composition of two successive operators $d^\nabla$ is a $0$-th order
differential operator, identified with an element $R^\nabla\in
\Omega^2(M,\End(E))$. In other words, for all $S\in\Omega^*(M,E)$ we have 
\[d^\nabla(d^\nabla S)= R^\nabla S
\]
where the product is induced from the canonical pairing $\End(E)\times E\to E$.
The tensor
$R^\nabla\in \Omega^2(M,\End(E))$ is called the \emph{curvature endomorphism} of
$\nabla$.

When $E$ has a nondegenerate symmetric bilinear pairing $h$ preserved by
$\nabla$, the
curvature endomorphism $R^\nabla$ will be skew-symmetric. As in \eqref{omegaa},
this endomorphism
determines via $h$ a double form
\begin{align*}
\omega_R\in\Omega^{2,2}=\Omega^2(M,\Lambda^2
E^*),&&\omega_R(X,Y)(e_1,e_2)=h(e_1,R^\nabla(X,Y)e_2)
\end{align*}
called the \emph{curvature form}. When no confusion can arise regarding $h$, we
will continue to denote the form 
$\omega_R$ by $R^\nabla$.

\subsection{The Pfaffian of the curvature}

For any real vector bundle $E\to M$ of even rank $2n$
endowed with a nondegenerate pairing $h$ and a compatible connection $\nabla$,
we are ready to define the Pfaffian form $\Pf_h(\nabla)$. When $E$ is oriented, 
multiplication by the volume form of $h$ defines for every $d\geq 0$ an
isomorphism
\begin{align*}
\Omega^d(M)\ni \alpha \mapsto \alpha\otimes\vol_h\in \Omega^{d,2n}.
\end{align*}
The inverse of this isomorphism is called the Berezin integral $\cB_h$ with
respect to $h$.
The \emph{Pfaffian form} of $\nabla$ is defined as
\[
\Pf_h(R^\nabla):=\tfrac{1}{n!}\cB_h(\omega_R^n)\in\Omega^{2n}(M)=\Omega^{2n,0}.
\]
When $h$ is implicit from the context, for simplicity we write $\Pf(\nabla)$ for $\Pf_h(R^\nabla)$.
Clearly, the Pfaffian form vanishes identically if the dimension of the base $M$
is smaller than the rank of $E$.
Even if $E$ is not oriented, we can still define the Pfaffian \emph{density}
of $R$, by using the above definition locally with respect to any of the two
possible orientations. 
This is well-defined because $\vol_h\otimes\vol_h$ is invariant under change of
orientation.

\begin{lem}\label{pszp}
If there exists a $\nabla$-parallel section $s\neq 0$ in $E$, then
$\Pf_h(R^\nabla)$ vanishes identically on $M$.
\end{lem}
\begin{proof}
If $\nabla s=0$, then $R^\nabla s=0$, thus $s\lrcorner \omega_R=0$, hence $s\lrcorner(\omega_R^n)=0$, so 
$\omega_R^n=0$.
\end{proof}

\begin{theorem}
The Pfaffian of $\nabla$ is a closed form on $M$.
\end{theorem}
\begin{proof}
From the second Bianchi identity, $d^\nabla \omega_R=0$, so by the Leibniz rule $d^\nabla (\omega_R^n)=0$.
Moreover, if $\alpha\in\Omega^{*,2n}$, then one can easily check that $d\cB_h(\alpha)=\cB_h(d^\nabla \alpha)$. Hence
\begin{align*}
d\Pf_h(\nabla)={}& \tfrac{1}{n!}d\cB_h(\omega_R^n)=\tfrac{1}{n!} \cB_h(d^\nabla
\omega_R^n)=0.\qedhere
\end{align*}
\end{proof}

With these preliminaries, we can now state the classical Gauss-Bonnet formula:
Let $(M,g)$ be a compact Riemannian manifold without boundary of dimension $2n$,
and $R\in\Omega^{2,2}$ the curvature of a metric connection on $TM$.
Then 
\begin{equation}\label{gbf}
\int_M \Pf(R)=(2\pi)^n \chi(M).
\end{equation}

The formula is valid even if $M$ is not orientable, in that case $\Pf(R)$ being
a density.  The reader interested in a 
short proof of this statement can skip the next section dealing with 
polyhedral manifolds and complexes thereof, as well as most of Section \ref{sec2}.

\begin{rem*}
The Pfaffian form on the tangent bundle was introduced by H.\ Hopf \cite{Hop}
motivated by geometric considerations that can be briefly summarized as follows:
The infinitesimal volume of the Gauss map of a compact hypersurface 
$M\subset \bR^{2n+1}$ equals the determinant of the 
second fundamental form, while the curvature of the hypersurface is, by the
Gauss equation, the square of the second fundamental form
in the sense of double forms. It follows that the pull-back on $M$ through the Gauss
map of the standard volume form from the sphere $S^{2n}$ equals 
the Pfaffian of the curvature. Hopf computed in this way the degree
of the Gauss map intrinsically in terms of the integral of the Pfaffian on $M$.
\end{rem*}

\section{Polyhedral manifolds}\label{Spm}

A somewhat informal notion of polyhedral manifold was used in \cite{AW}. 
We give here a rigorous definition together with an extension to a larger
category, that of polyhedral complexes.

Polyhedral manifolds extend the notion of manifolds with corners \cite{Mel}. 
Since for manifolds with embedded corners the results of the
current section become largely obvious, the reader interested only in such manifolds
can proceed directly to Section \ref{sec2}.

\subsection{Linear polyhedra}

\begin{definition*}
Let $V$ be a real vector space of dimension $n$, and $\SVs=(V^*\setminus
\{0\})/\Rz^*_+$ the dual sphere, consisting of non-zero linear forms on $V$
defined up
to a positive constant. Let $A\subset \SVs$ be a finite set.
The \emph{open linear polyhedron} in $V$ defined by $A$ is the cone
\[P_A=\{x\in V; \alpha(x)<0, (\forall)[\alpha]\in A\}. 
\]
\end{definition*}
A non-empty linear polyhedron is always unbounded according to this definition, since
it is invariant
by dilations with positive factors.
The condition $\alpha(x)<0$ is the same for every representative
$\alpha\in[\alpha]$, i.e., 
it is invariant under rescaling of $\alpha$ by a positive constant.

A \emph{closed linear polyhedron} is the closure in $V$ of an open linear polyhedron. The
closure of $P_A$
is clearly
\[\overline{P}_A=\{x\in V; \alpha(x)\leq 0, (\forall)\alpha\in
A\}\]
and note that according to this definition, the interior in $V$
of a closed polyhedron is nonempty.

Some of the linear forms defining a linear polyhedron $P_A$ may be redundant, so it is
natural to consider \emph{minimal} sets of such defining forms. 

\begin{lem}
If a set $A\subset \SVs$ defines a non-empty open linear polyhedron in $V$, there
exists
a unique minimal set $A'\subset A$ defining the same linear polyhedron, i.e.,
$P_{A'}=P_A$. Two minimal sets $A,A'\subset\SVs$ define the same polyhedra in
$V$ 
if and only if they are equal.
\end{lem}

Let $P_A$ be a nonempty open linear polyhedron defined by a minimal set of linear forms
$A$. For every $\alpha\in A$, the \emph{hyperface} $P_A^\alpha$ is the open
linear polyhedron
inside the vector space $\ker(\alpha)\subset V$ defined by the relations
$\{\beta_{|\ker(\alpha)};\beta\in A, \beta\neq\alpha\}$. 
By minimality, this linear polyhedron is non-empty, has dimension $n-1$, and
sits inside the closure of $P_A$. The linear polyhedron $P_A$ has thus as many
hypersurfaces as the
cardinality of $A$. Note that the set of defining forms for $P_A^\alpha$
indexed by
$A\setminus\{\alpha\}$ may be non-minimal. 

Inductively, one defines the faces of depth (or codimension) $l\geq 2$ of a
linear polyhedron $P$ as the
hyperfaces of the faces of depth $l-1$ of $P$. A closed linear polyhedron is thus
decomposed into the disjoint union of its open faces. We denote 
\[\cF(P)=\bigcup_{l\geq 1}\cF^{(l)}\]
the set of all faces of $P$ of codimension at least $1$.

Recall the Minkowski-Weyl theorem \cite{Wep,bero}: every open linear polyhedron $P$ can
be described alternately as the set of linear combinations with positive
coefficients of some generating vectors $v_1,\ldots v_k$:
\[P=\{c_1v_1+\ldots +c_kv_k; c_j>0,(\forall) j=1,\ldots,k\}.
\]
A minimal set of such generating vectors is unique up to rescaling by
positive constants. Conversely, every finite set of vectors in $V$ spans a linear
polyhedron by positive linear combinations as above. (This linear polyhedron could be
the whole of $V$, corresponding to the empty set of linear forms $A=\emptyset$).

If $P=P_A$ is an open linear polyhedron in $V$ defined by a finite set of forms
$A\subset\SVs$, the \emph{dual polyhedron $P^*$} is the linear polyhedron inside the
dual space $V^*$ defined as the positive linear span of the vectors $\alpha\in
A$. For every face $F\subset \overline{P}$, the \emph{conormal space} $N^*F$ is
the space of forms in $V^*$ which vanish on $F$ (or equivalently, on its linear
span). The intersection of $N^*F$ with $P^*$ is a face of $P^*$.

\subsection{Polyhedra}

\begin{definition}
An \emph{open polyhedron} in a vector space $V$ is the set of points $v\in V$ satisfying
the inequalities
\begin{align}\label{defp}
\alpha_j(v)< a_j,&&j=1,\ldots,k
\end{align}
for some $\alpha_1,\ldots,\alpha_k\in V^*$ and $a_1,\ldots,a_k\in\Rz$. 
\end{definition}

Closed polyhedra and their faces are defined as in the case of linear polyhedra. 
A \emph{polytope} is a compact polyhedron. According to our definition, 
the interior in $V$ of a closed polyhedron is always nonempty, and its dimension is $\dim(V)$.

\begin{definition}
The \emph{conormal outer cone} $F^\out$ of an open face $F\subset \overline{P}$
in a polyhedron $P\subset V$ is 
the set of forms $\alpha\in V^*$ such that for every $v\in F$ and $v_0\in P$,
$\alpha(v-v_0)>0$. 
\end{definition}

If the polyhedron $P$ is defined by \eqref{defp} and $F\subset\overline{P}$ 
is the open face determined by
\[
F=\{v\in V; \alpha_1(v)=a_1,\ldots,\alpha_l(v)=a_l, \alpha_{l+1}(v)<a_{l+1},\ldots,\alpha_k(v)<a_k\}
\]
then $F^\out$ can also be described as the open cone
\[F^\out= 
\{c_1\alpha_1+\ldots+c_l\alpha_l; c_j>0,(\forall) j=1,\ldots,k\}
\]
 so in particular $F^\out$ is a linear polyhedron in $N^*F\subset V^*$. If we fix $v\in F$ and $v_0\in P$, take
\begin{equation}
F^\out_1 = \{\alpha\in F^\out; \alpha(v-v_0)=1\}.
\end{equation}
Then $F^\out_1$ is an open polytope independent of $v$, and independent of $v_0$ up to a canonical isomorphism. Moreover,
$\Rz_+^* F^\out_1=F^\out$, or in other words $F^\out$ is a cone with base $F^\out_1$.

A \emph{subcomplex $C\subset P$} in a closed polyhedron $P$ is the union of some closed
faces of $P$. 

\begin{definition}
Fix $v_0$ in the interior of a polyhedron $P$.
Let $P^\out\subset V\times V^*$ be the disjoint union
\[P^\out=\bigcup_{F\in\cF(P)} F\times F^\out.
\]
Define also $P^\out_1$ as
\[
P^\out_1=\bigcup_{F\in\cF(P)} F\times F^\out_1.
\]
\end{definition}

\begin{lem}\label{pouc}
If $P$ is a linear polyhedron in $V$ then the sets $P^\out$ and $P^\out_1$ (defined with respect to some fixed $v_0\in P$) are polyhedral subcomplexes in $P\times P^*\subset V\times V^*$.
\end{lem}
\begin{proof}
Clearly $P^\out$ is a subcomplex of the product polyhedron $P\times P^*$. 
As for $P^\out_1$, it is the intersection of $P^\out$ with the affine hyperplane $\{(v,\alpha);\alpha(v_0)=-1\}$.
\end{proof}

\subsection{Polyhedral manifolds and complexes}

It is natural to define a \emph{topological polyhedral manifold} $X$ of
dimension $n$ as a separated topological
space locally homeomorphic to some closed polyhedron of dimension $n$, i.e.,
every point $x\in X$ has
a neighborhood homeomorphic to an open set in a closed polyhedron (depending on
$x$) of dimension $n$. These local homeomorphisms are called \emph{charts},
and a collection of charts covering $X$ is an \emph{atlas}.
Since every polyhedron is locally homeomorphic 
to a linear polyhedron by an affine transformation, we could have used \emph{linear} 
polyhedra as model spaces in this definition.

\begin{example}\label{esp}
Fix an inner product on $V$. An open \emph{spherical polytope} in $V$ is the
intersection of an open linear polyhedron $P\subset V$ with the unit sphere in $V$. 
The closure of a spherical polytope is a polyhedral manifold diffeomorphic to
the polytope 
\[P_1=\{v\in P; \alpha_0(v)=-1\}
\]
for some fixed $\alpha_0$ in the dual polyhedron $P^*$. 
The diffeomorphism is given by radial projection.
\end{example}

A \emph{topological polyhedral complex} inside a polyhedral manifold $M$ is a
space $C$ so that the pair $(M,C)$ is locally homeomorphic to some subcomplex
inside some closed polyhedron. More precisely, for every $x\in C$
there exists a polyhedron $P_x$, a subcomplex $C_x\subset P_x$, and a
homeomorphism, called \emph{chart of complexes},
from a neighborhood $D\subset P$ of $x$ to an open set $U\subset P_x$, such that
$U\cap C$ is mapped homeomorphically onto $U\cap C_x$.

Let $P\subset V$, $P'\subset V'$ be open polyhedra and $C\subset \overline{P},
C'\subset\overline{P'}$ subcomplexes. 
A map $f:D\subset C\to C'$ is called \emph{smooth}
if for every $x\in C$ there exists an open neighborhood $U$ of $x$ in $V$ and a
smooth map $F:U\to V'$ extending $f|_{U\cap C}:U\cap C\to C'$ such that $F(U\cap P)\subset P'$.

\begin{definition*}
A \emph{polyhedral manifold} (respectively a \emph{polyhedral complex}) is 
is a topological polyhedral manifold (respectively a topological polyhedral complex) 
endowed with a smooth atlas.
\end{definition*}

A point of depth $l\geq 0$ in a polyhedral manifold $M$ is a point mapped to a
point of depth $l$ through
a chart (hence through any chart). A \emph{open face of codimension $l$} of $M$
is a connected component 
of the set of points of depth $l$. Such a face is clearly a smooth manifold of
dimension $n-l$. However, its closure
is in general \emph{not} a polyhedral manifold of the same dimension! 

For a face $Y$ inside a polyhedron $M$, we denote by $\cF^{(l)}(Y)$ the set of open
faces of $M$ which have codimension $l\geq 0$ in $\overline{Y}$. 

\begin{definition}\label{defreg}
A polyhedral manifold $M$ is \emph{regular} if the closure of every open face of
$M$ is again a polyhedral manifold.
\end{definition}

Linear polyhedra and spherical polytopes are examples of regular polyhedral
manifolds; the region in the plane bounded by a smooth segment self-intersecting
orthogonally in its end-points is an example of non-regular polyhedral surface. Regular 
polyhedral manifolds are an extension of the notion of manifolds-with-corners 
with embedded faces.
The polyhedral manifolds considered in \cite{AW} seem to be regular, although the
authors are imprecise on this aspect.

\subsection{Outer spheres and the outer cone complex}\label{oc}

Let $x$ be a point in a face $Y$ of a polyhedral manifold $M$. An
\emph{interior vector} at $x$ is the tangent vector in $0$ 
to a smooth curve $c:[0,\epsilon)\to M$ with $c(0)=x$. 

A Riemannian metric on a polyhedral manifold $M$ is a metric which, in any chart
modeled by a polyhedron in a vector space $V$, extends
to a smooth metric on an open set of $V$. Riemannian metrics can be constructed
on any polyhedral manifold using partitions of unity.

\begin{definition}
Let $x\in M$ be a point in the Riemannian polyhedral manifold $M$, and let
$Y$ be the unique open face of $M$ of which $x$ is an interior point.
The \emph{outer cone at $x$}, denoted $C^\out_x Y$, is the set 
of those vectors in $T_xM$ whose inner products with every 
interior tangent vector at $x$ are non-positive.
The \emph{outer sphere} $S^\out_x Y$
is the set of unit vectors in the outer cone at $x$. 
\end{definition}

Every vector $V\in S^\out_x Y$ is orthogonal to $Y$, so $S^\out_x Y$ is a subset
of the normal sphere to $Y$ at $x$ (otherwise, the projection
of $V$ on $Y$ would have positive inner product with $V$). Thus $S^\out_x Y$ is
a spherical polytope inside $N_xY$, the orthogonal complement
in $T_XM$ to $T_xY$.

\begin{prop}\label{lemfb}
Let $((M,g^M)$ be a Riemannian polyhedral manifold of dimension $n$, $Y\subset
M$ an open face of codimension $l$, and $a\in Y$ a fixed point. Then 
\[ S^\out Y= \bigsqcup_{x\in Y} S^\out_xY\subset TM
\]
is a flat, locally trivial bundle over $Y$ with fiber type the polytope
$Y^*=S^\out_a Y$. If $M$ is regular, then $S^\out Y$ 
is globally trivial, $S^\out Y\simeq Y\times Y^*$ .
\end{prop}
\begin{proof}
Let $a\in M$. There exists a connected domain $M\supset D\ni a$, a chart $\phi:D
\to \overline{P}\subset \Rz^n$ with values 
in a closed polyhedron $\overline{P}$, and a Riemannian metric $g^P$ on $\Rz^n$
whose pull-back through $\phi$ is $g^M$. 
Let $Y$ be the unique open face of $M$ containing $a$, and denote by $l$ its
dimension. Then $D\cap Y$ is mapped into a face $F\subset P$ of the same
codimension. Clearly $\phi_*:TM\to TP$ is a bundle isomorphism over its image,
and thus for every $x\in D$
its restriction to the outer spheres $\phi_x:S^\out_x Y\to S^\out_{\phi(x)}F$ is
a linear isomorphism of spherical polytopes. To show that $S^\out Y$ is locally
diffeomorphic to a product, it is thus enough to show that $S^\out F$ is
canonically isomorphic to the product bundle over $F$ with fiber type
$S^\out_bF$. 

The image of $C^\out F$ through the isomorphism $TP\to T^*P$ induced by $g^P$
is just $F\times F^\out$. 
We have seen in Example \ref{esp} that after fixing an interior point in $P$,
the bundle of spheres in $F\times F^\out$
becomes canonically diffeomorphic to the product polytope $F\times F^\out_1$.

Assume now that $M$ is regular. The closure of every open face is then
embedded in $M$, and hence the unit outer normal vector field to every open hyperface
extends continuously to the closure.
It follows that every closed face $Y$ is a connected component of an intersection
of hyperfaces. For all $x\in Y$, 
the generating vectors of the spherical polytope $S^\out_xY$ are the unit normals
to distinct hyperfaces of $M$, and they trivialize $S^\out Y$ as in Example \ref{esp}.
\end{proof}

The contribution of a face $Y\subset M$ in the Gauss-Bonnet formula will turn out
to be given by an integral on the total space of the bundle of outer spheres $S^\out Y$,
or equivalently the integral on $Y$ 
of a transgression of the Pfaffian of order $\codim(Y)$.

\begin{prop}\label{mopc}
Let $(M,g^M)$ be a Riemannian polyhedral manifold of dimension $n$. Then the set
of outer-pointing unit tangent vectors, denoted $M^\out$, is a polyhedral
complex of dimension $n-1$ inside $TM$. The open top-dimensional faces of
$M^\out$ are the interiors of the outer sphere bundles 
$S^\out Y$, where $Y$ spans all the open faces of $M$ of codimension $l\geq 1$. 
\end{prop}
\begin{proof}
Suppose first that $M$ is the closure of an open polyhedron in $\Rz^n$ defined
by a minimal set of linear forms $\alpha_1,\ldots,\alpha_p\in A$.
If $Y\subset M$ is a face of codimension $l$, let $A_Y\subset A$ be the set of
those $\alpha_j$'s which vanish along $Y$. After relabeling, we can assume that
$A_Y=\{\alpha_1,\ldots,\alpha_k\}$.
Proposition \ref{lemfb} tells us that $C^\out Y$ is isomorphic to $Y\times Y^*$, where
the polytope $Y^*$ is the outer sphere at some fixed point $a\in Y$. Moreover,
$S^\out_y Y$ is spanned as a spherical polytope by the unit vectors
$V_1(y),\ldots,V_k(y)$ dual via $g^M$ to $\alpha_1\ldots,\alpha_k$. It follows that
$S^\out_y Y$ is diffeomorphic to the conormal outer cone $Y^\out$ under the difeomorphism
$TM\to T^*M$ induced by the metric $g^M$. Hence $C^\out Y$, respectively $S^\out Y$
are diffeomorphic 
to $Y\times Y^\out$, respectively to $Y\times Y^\out_1$. Thus the claim follows from
Lemma \ref{pouc}.

In general, $M$ is only locally diffeomorphic to a polyhedron $P$, thus $M^\out$ 
is locally diffeomorphic to a subcomplex in $TP$, which by definition means that
it is a polyhedral complex.
\end{proof}

\subsection{Boundaries of polyhedral complexes and the Stokes formula}

It is straightforward to define smooth differential forms, the exterior
derivative, and restriction of forms to faces of polyhedral manifolds. 
We prove below that Stokes formula also
continues to hold on polyhedral complexes, once we properly define the 
boundary of a face. 

Let $M$ be a polyhedral complex of dimension $k\geq 1$
and assume that we fix an orientation on all the faces of $M$ of dimension $k$ and $k-1$.
Let $F\subset M$ be an open face of dimension $k-1$. We define an integer $\mu_M(F)$,
the \emph{multiplicity} of $F$ in $M$, counting how many times $F$ appears 
as the oriented boundary of faces of $M$, as follows: take $x\in F$ and a 
connected chart in $M$ near $x$,
$\phi:D\to P$, mapping a neighborhood of $x$ 
into the $k$-skeleton of a polyhedron $P$. The chart induces orientations 
on the $k$-dimensional faces $Y_1,\ldots,Y_s$ of $P\cap \phi(D)$ whose closure 
contains the image of $x$, and also on the unique 
face $F'\subset P\cap\phi(D)$ of dimension $k-1$ containing $\phi(x)$. 
For $j=1,\ldots,s$ let $\nu_j$ be a vector field along $F'$ pointing inside $Y_j$. Let
$\mu_x(F',Y_j)\in\{\pm1\}$ be $1$ if the orientations on $F'$ and $Y_j$ are compatible
(i.e., $\nu_j,e_1,\ldots,e_{k+1})$ is a negatively oriented frame in $T_xY_j$ whenever
$(e_1,\ldots,e_{k+1})$ is a positively oriented frame in $T_xF'$) and $-1$ otherwise.
Define
\[\mu_x(F,M)=\sum_{j=1}^s \mu(F',Y_j)\in\cz.\]
This quantity is locally constant on the connected face $F$, hence it is a constant,
denoted by $\mu(F,M)\in \cz$. We define the $(k-1)$-boundary of $M$ as the formal sum
\[\partial_{k-1}(M)= \sum_{F\in\Sk^{k-1}M} \mu(F,M)\cdot F.
\]

Note that for a polyhedral manifold $M$ of dimension $n$ with orientable interior, 
$M^\out$ is a polyhedral complex with a natural orientation on the $n-1$-dimensional 
faces, and
\[\partial_{n-2} M^\out  = 0.
\]

\begin{lem}\label{thSt}
Let $M$ be an oriented polyhedral complex of dimension $k$
and $\omega\in\Omega^*(M)$ 
a compactly supported form. Then
\[\int_M d\omega = \int_{\partial_{k-1} M} \omega.
\]
\end{lem}
\begin{proof}
By local
considerations involving partitions of unity, it is enough to prove the claim when the form 
$\omega$ is supported in a chart domain. We may therefore  assume that $M$ is a 
subcomplex of dimension $k$ in a polyhedron $P$, and moreover that $\omega$ is supported in a small ball 
which only intersects those faces passing through its center. On the
intersection of this ball with each open $k$-dimensional face $Y$ we compute the integral
of $d\omega$ using the usual Stokes formula, 
obtaining the integral of $\omega$ on the $k-1$-dimensional faces $F$ bounding $Y$, with a sign depending on
the compatibility between the orientations on $Y$ and $F$.
When summing over all $Y$ for a fixed $F$, we eventually get $\mu(F,M)$ times the integral of $\omega$ on $F$.
\end{proof}

\section{Transgression forms}\label{sec2}

Let $X$ be an $l$-dimensional polytope
(see Section \ref{Spm} for the definition). Let $E\to M$ be a
real vector bundle of rank $2n$ endowed with a pseudo-metric and a 
compatible connection over a smooth manifold $M$ of arbitrary dimension.

Let $V$ be a family indexed by $X$ of unit-length sections in $E$ over $M$,
i.e., a section $V:X\times M\to \pi_2^* E$ such that $h(V,V)=1$. 
In particular, for every $x\in X$, $V(x,\cdot)$ is a section in $E$, so $V$ can
be viewed as a map $V\in C^\infty(X,\Omega^{0,1}(M,E))$
(remember that $E$ is identified with $E^*$ using $h$). The connexion $\nabla$
in $E$ acts of such section-valued maps, and
$\nabla V\in C^\infty(X,\Omega^{1,1}(M,E))$ is a family of $(1,1)$-forms indexed
by $X$. Let $d^X V$ be the differential on $X$ of the $\Omega^{0,1}(M,E)$-valued
function $V$. In local coordinates, 
\begin{align*}
d^X V:={}&\sum_{j=1}^l dx^j\otimes \partial_{x_j} V\in
\Omega^{1}(X,\Omega^{0,1}(M,E)),\\
(d^X V)^l= {}& l! dx^1\wedge\ldots\wedge dx^l\otimes \partial_{x_1}
V\wedge\ldots\wedge\partial_{x_l}V\in\Omega^{l}(X,\Omega^{0,l}(M,E)).
\end{align*}
For integers $l,k,n$ satisfying $0\leq l\leq 2k+1\leq 2n-1$, define a universal
constant $c(n,k,l)$ by
\[c(n,k,l)= \tfrac{2^{k} k!}{(n-1-k)!(2k+1-l)!}\in\mathbb{Q}.\]
\begin{definition}\label{deftr}
For $l\geq 0$, we define the $(l+1)^{\text{th}}$ transgression of the Pfaffian
with respect to the family $V:X\to\Omega^{(0,1)}(M)$ by
\[\cT^{(l+1)}_{V}=\sum_{l\leq 2k+1\leq 2n-1} \tfrac{c(n,k,l)}{l!} \int_X \cB 
\left[V (d^X V)^l \left(\nabla V\right)^{2k+1-l}R^{n-1-k} \right]\in
\Omega^{2n-l-1}(M).\]
\end{definition}
Here $(d^X V)^l\in \Omega^l(X)\otimes \Omega^{0,l}(M,E)$,
$R\in\Omega^{2,2}(M,E)$, and $\nabla V\in C^\infty(X,\Omega^{1,1}(M,E))$,
so $\cB$ is applied to a volume form on $X$ tensored with a form of degree $2n-l-1$
on $M$. 

For $l=0$, the first transgression of the Pfaffian of the Levi-Civita connection
on the sphere bundle of a Riemannian manifold
was introduced by Chern \cite{Ch1}.

\subsubsection*{Functoriality}

Like the Pfaffian, the transgression forms are functorial: Let $E\to M$ be a
vector bundle with metric and connection,
and $\Phi:N\to M$ a smooth map. We equip $\Phi^*E$ with its pull-back metric and
connection. 
For every family $V$ of unit sections in $E$ indexed by $X$, we get a family
$\Phi^*V$ of unit sections in $\Phi^*E$.
In this framework,
\[\cT^{(l+1)}_{\Phi^*V} = \Phi^* \cT^{(l+1)}_{V}\in \Omega^{2n-l-1}(N).\]

When $V$ can be inferred from the context, 
we will write $\cT^{(l+1)}_{Y}$ instead of $\cT^{(l+1)}_{V_Y}$ for
the transgression form with respect to the restriction of $V$ along a 
subcomplex $Y$ of $X$. By functoriality, this notation does not lead to confusion.

We are now ready to prove our main result about transgressions.

\begin{theorem} \label{thT}
Let $E\to M$ be a vector bundle endowed with a semi-riemannian
metric $h$ and a compatible connection $\nabla$, 
and $V$ a family of unit-length sections in $E\to M$ indexed by an
oriented polytope $X$ of dimension $l$. Then
\[d \cT^{(l+1)}_X=\begin{cases} -\cT^{(l)}_{\partial X}& \text{for $l=\dim X\geq
1$},\\
-\Pf(\nabla) & \text{for $X=*$, i.e., $l=\dim(X)=0$},
\end{cases}
\]
where $\cT^{(l)}_{\partial X}\in\Omega^{2n-l}(M)$ is the sum of the
transgression forms corresponding to the oriented hyperfaces of $X$, 
i.e., the transgression corresponding to the boundary cycle of $X$.
\end{theorem}
\begin{proof}
Let $\pi:\SE\to M$ be the locally trivial bundle of unit (pseudo-)spheres in $E$
with respect to $h$. The tangent bundle to $\SE$ contains the vertical tangent
bundle to the fibers. There is a natural horizontal complement to this vertical
bundle, defined by using the connection $\nabla$: the horizontal lift of
a path $\gamma:\bR\to M$ at a point $v\in S_{\gamma(0)}M$ is the parallel
transport of 
$v$ along $\gamma$. Thus $\nabla$ induces a splitting of $T\SE$ as 
\[T\SE= T^\verti \! \SE\oplus \pi^* TM.
\]
In the vector bundle $\pi^*E\to \SE$ we have the tautological section $s$ of
$h$-length $1$, defined by $s_v:=v\in E_{\pi(v)}=(\pi^*E)_v$.
This connection preserves the pull-back metric $h$,
still denoted by the same symbol.
\begin{lem}
Let $\nabla^1=\pi^*\nabla$ be the pull-back connection in $\pi^*E\to\SM$. Then
\begin{equation}\label{naus}
\nabla^1 s= I_{T^\verti \SE}
\end{equation}
in the sense that $\nabla^1_U s= U$ for every $U\in T_v^\verti \!
\SE=T_v(S_vM)\subset E_v$. 
\end{lem}
\begin{proof}
Essentially by definition, the canonical section $s$ is parallel 
in horizontal directions with respect to $\nabla^1$, so $\nabla^1s$
is a vertical double form.
Also by definition, the pull-back connection is trivial in vertical directions,
so we compute $\nabla^1 s = ds$ on each vertical pseudo-sphere $S_vE$, where 
$s$ becomes a map from $S_vE$ to the fixed vector space $E_v$. 
\end{proof}

The idea of computing the Pfaffian of $\nabla^1$ (going back to Chern
\cite{Ch1}) is to modify $\nabla^1$ on $\SE$ so that $s$ becomes parallel. For
this, define
a vertical $(1,2)$-double form $\alpha\in \Omega^1(\SE,\Lambda^2\pi^*E)$ by
\begin{align*}
\alpha=s\cdot \nabla^1 s,&&\alpha(U)=s\wedge U\in \Lambda^2\pi^*E
\end{align*}
for every $U\in T^\verti \! \SE$, and $\alpha(H)=0$ for $H\in T^\hor \! \SE$
horizontal. 
Denote by $A$ the skew-symmetric endomorphism-valued $1$-form
associated to $\alpha$ via $h$ as in \eqref{omegaa}:
\begin{align}\label{as}
\alpha(U)(V,W)=\langle V,A(U)W\rangle_h,&&(\forall) U\in T\SE, (\forall)
V,W\in\pi^*E,
\end{align}
so $\alpha=\omega_A$ using the notation from \eqref{omegaa}. (We identify $E$
with $E^*$ via the musical isomorphism 
in terms of $h$, thus $\Lambda^2\pi^*E\simeq \Lambda^2\pi^*E^*$). Then 
\[\nabla^1 s=-A s.\]
For $t\in\bR$ set 
\[\nabla^t:=\nabla^1+(1-t)A.
\]
Since $A$ is skew-symmetric, $\nabla^t$ clearly preserves $h$.
We compute $\nabla^t s= -t A s=tI_{T^\verti \! \SE}$, hence $\nabla^0s=0$ and
thus, by Lemma \ref{pszp}, $\Pf_h(\nabla^0)=0$. We shall recover the Paffian of
$\nabla^1$ as the integral from $0$ to $1$ of the $t$-derivative of
$\Pf(\nabla^t)$.

Consider the connection $D$ on the pull-back bundle $\pi^* E\to \bR\times \SE$
defined by
\begin{equation}\label{defD}
D:=dt\otimes \partial_t(\cdot)+\nabla^t=dt\otimes
\partial_t(\cdot)+\nabla^1+(t-1)s\cdot\nabla^1 s.
\end{equation}
This connection also preserves the (pull-back of the) metric $h$. 

Let $V:X\times M\to \SE$ be a family of unit sections in $E$ (i.e., for every
$x\in X, p\in M$,
$V(x,p)\in E_p$ is a unit-length vector). Consider the smooth map
\begin{align*}
\Phi:\bR\times X\times  M\to \bR\times \SE,&& \Phi(t,x,p)=(t,V(x,p))
\end{align*}
and let $\Phi^* D$ be the pull-back connection in the bundle $\Phi^*\pi^*E=
\pi_3^* E$
(where $\pi_3:\bR\times X\times M\to M$ is the projection on the third factor).

By the naturality of curvature, we have
\[\Phi^*\Pf_h(D)= \Pf(\Phi^* D)\in\Omega^{2n}(X\times\bR\times M).
\]
Integrating this Pfaffian in the $X\times \bR$ variables, we set 
\begin{equation}\label{TX}
T_X:= \int_{[0,1]\times X} \Pf(\Phi^*D)\in \Omega^{2n-l-1}(M).
\end{equation}
We shall show below that $T_X$ equals the transgression $\cT^{(l+1)}_{X}$ from
Definition \ref{deftr}; for the time being, we compute
\[d T_X= \int_{[0,1]\times X} d^M \Pf(\Phi^*D)\in \Omega^{2n-l}(M).
\]
The Pfaffian form is closed on $\bR\times X\times  M$, so by Stokes formula on
the polyhedral manifold $[0,1]\times X\times  M$,
\begin{align}
d T_X= {}&\int_{[0,1]\times X} -dt\wedge \partial_t \Pf(\Phi^*D)-d^X
\Pf(\Phi^*D)\nonumber \\
={}&\int_{\{0\}\times X} \Pf(\Phi^*D)- \int_{\{1\}\times X} \Pf(\Phi^*D)
-\int_{[0,1]\times\partial X}\Pf(\Phi^*D)\label{ttr}
\end{align}
where $\partial X$ is the oriented sum of hyperfaces of $X$.
Now $D_{|\{0\}\times \SE}=\nabla^0$. We have seen above that $\Pf(\nabla^0)=0$
because there exists a non-zero parallel section for $\nabla^0$ on $\SE$, so by
naturality of the Pfaffian, $\Pf(\Phi^*D)=0$ on $\{0\}\times X\times M$.

Similarly, $D_{|\{1\}\times \SE}=(0,V)^*\pi^*\nabla=\pi_3^*\nabla$, so
$\Pf(\Phi^*D)=\pi_3^*\Pf(\nabla)$ on $\{1\}\times X\times M$. This pull-back 
form does not contain any $dx^j$ (where $x_1,\ldots,x_l$ are the euclidean
variables on $X$). Hence for $l>0$ the integral on $\{1\}\times X$ of the second
term also vanishes, while for $l=0$ (i.e., when $X$ is a point) it reduces to
$-\Pf(\nabla)$.

By definition, the third term from \eqref{ttr} is the sum of the transgressions
$T_F$ corresponding to
the oriented hyperfaces $F$ of $X$, which we denote $T_{\partial X}$.

In order to show that $T_X$ defined in $\eqref{TX}$ is the same as the
transgression $\cT^{(l+1)}_X$ from Definition \ref{deftr}, we must compute in
more detail the Pfaffian of $\Phi^* D$. From \eqref{defD},
\begin{align*}
R^D=\pi^* R^\nabla -(t-1)\pi^*R^{\nabla} s\wedge s
+\tfrac{1-t^2}{2}\nabla^1s\cdot\nabla^1 s +dt\otimes s \cdot\nabla^1 s.
\end{align*}
The $\cdot$ products above are in the sense of double forms. Using the somewhat
imprecise
but suggestive notation $\nabla$ for the connection $\pi_3^*\nabla$ over 
$\bR\times X\times  M$, we get for the curvature of $\Phi^*D$:
\begin{align*}
\Phi^*R^D=R^\nabla -(t-1)R^{\nabla} V\wedge V +\tfrac{1-t^2}{2}(d^X V+\nabla
V)^2 +dt\otimes V \cdot(d^X V+\nabla V).
\end{align*}
We proceed to analyze the $n^{\text{th}}$ power of this double form inside the
space of
double forms $\Omega^{2n,2n}(\bR\times X\times  M,\pi_3^*E)$. Since double forms
of even bi-order
form a commutative algebra, we treat this power as a homogeneous polynomial
of degree $n$ in these four terms of degree $(2,2)$.
We are only interested in those monomials containing the volume form of
$\bR\times X$, 
and clearly those terms must contain precisely once the form $dt$.
Thus, the term $dt\otimes V \cdot(d^X V+\nabla V)$ appears precisely once, and
so 
$(t-1)R^{\nabla} V\wedge V$ does not contribute at all (since it contains
$V$
which already appeared in the former term, while for a monomial of top fiber
degree $2n$ to be nonzero, the section $V$
may occur at most once.) In conclusion of this discussion,
those
monomials from $(\Phi^*R^D)^n$ containing the volume form of $\bR\times X$ are
contained in
\[
ndt\otimes V \cdot(d^X V+\nabla V) 
\left[R^\nabla +\tfrac{1-t^2}{2}(d^X V+\nabla V)^2\right]^{n-1}.
\]
Using the binomial formula, we write the above as
\[
ndt\otimes V \cdot\sum_{k=0}^{n-1}\tbinom{n-1}{k}\tfrac{(1-t^2)^k}{2^k} (d^X
V+\nabla V)^{2k+1} 
\left(R^\nabla \right)^{n-k-1}.
\]
Apply the binomial formula to the term $(d^X V+\nabla V)^{2k+1}$, and retain
only the term of top degree $l$ in the $X$
variables, since we need to integrate over $X$. Hence 
$T_X$ from \eqref{TX} is computed as
\[
\sum_{k=\left\lceil\frac{l+1}{2}\right\rceil}^{n-1}\int_0^1 (1-t^2)^k dt\cdot
n\tbinom{n-1}{k}2^{-k} \tbinom{2k+1}{l} \tfrac{1}{n!}
\int_X \cB\left[V\left(d^X V\right)^l (\nabla V)^{2k+1-l}\left(R^\nabla
\right)^{n-k-1}\right].
\]
This gives precisely the transgression from Definition \ref{deftr}, since
\[
\int_{0}^1 (1-t^2)^{k}dt=\tfrac{2k(2k-2)\ldots 2}{(2k+1)(2k-1)\ldots 3\cdot
1}=\tfrac{(2k)!!}{(2k+1)!!}.\qedhere
\]
\end{proof}

\section{The Gauss-Bonnet formula on polyhedral manifolds} \label{sec3}

\subsection{The Allendoerfer-Weil formula in even dimensions}

Let $M$ be a compact $C^\infty$ polyhedral manifold of dimension $2n$ endowed
with a Riemannian metric $g$.
Let $Y\subset M$ be a face of codimension $l\geq 0$, and $N_Y\subset
TM_{|Y}$ the normal bundle of $Y$ inside $M$ with respect to $g$. 
The second fundamental form of this inclusion is the bilinear map
\begin{align*}
A:TY\times TY\to N_Y,&& A(U,W)=\nabla^M_U W-\nabla^Y_U W,
\end{align*}
hence $A$ is a $N_Y$-valued double form of degree $(1,1)$ on $Y$. We construct
from $A$ its dual, a smooth section $A^*$ on $N_Y$ in the pull-back from $Y$
of the bundle of double forms:
\begin{align}\label{secff}
A^*\in C^\infty(N_Y, \Lambda^{1,1}(Y)),&& A^*(V)=g(V,A).
\end{align}
For any $B,C\in \Omega^1(Y,T^*Y\otimes N_Y)$ pure tensors of the form
\begin{align*}
B=b_1\otimes b_2\otimes \nu_1,&& C=c_1\otimes c_2\otimes \nu_2
\end{align*}
define the partial contraction with $g$ by
\begin{align}\label{contra}
g(B,C)=g(\nu_1,\nu_2)(b_1\wedge c_1)\otimes (b_2\wedge c_2)\in\Omega^{2,2}(Y).
\end{align}
This definition allows us to define by linearity $g(A,A)\in\Omega^{2,2}(Y)$. 

We are now ready to prove the extension of the Allendoerfer-Weil formula \cite{AW} for
the Euler characteristic of a compact Riemannian polyhedral manifold.

\begin{theorem}\label{awt}
Let $M^{2n}$ be a compact Riemannian polyhedral manifold. Then
\begin{align*}
(2\pi)^n\chi(M)-\int_M\Pf(R)={}&\sum_{l=1}^{2n}\sum_{k=\left\lceil\frac{l}{2}
\right\rceil}^{n}
\tfrac{(-1)^l 2^{k-1} (k-1)!}{(n-k)!(2k-l)!}\\
{}&\sum_{Y\in \cF^{(l)}(M)} \int_{S^\out Y} \cB_Y\left[(R^Y-\tfrac{1}{2}
g(A,A))^{n-k} (A^*)^{2k-l}\right] |dg|.
\end{align*}
\end{theorem}
Here $|dg|$ is the family of spherical volume forms induced by $g$ on the
fibers  of the normal sphere bundle $SY\to Y$, while the Berezin integral 
produces  a volume form on $Y$. The above integral can thus be considered 
either as an integral on the total space of $S^\out Y$, or 
(using Fubini's theorem) as an iterated integral, first along the fibers of 
$S^\out Y\to Y$ and then on $Y$. The symbol $\left\lceil\frac{l}{2}\right\rceil$ 
denotes the smallest integer greater than or equal to $l/2$. We use the 
convention $0!=1$. Also, the $0^{\text{th}}$ power of a tensor like 
$A^*$ (for $2k=l$) or $R^Y-\tfrac{1}{2} g(A,A)$ (for $k=n)$ is understood
to be always $1$, regardless of the vanishing of the tensor in question.
\begin{proof}
We first give the argument under the assumption that the polyhedral manifold 
$M$ is regular (Definition \ref{defreg}), thus recovering the main result of \cite{AW}.
The general case requires some additional combinatorial properties 
of the outer cone complex and will be treated in Section \ref{s54}.

We apply successively the transgression formula from Theorem \ref{thT} to the
vector bundle $TM$ restricted to the various faces of $M$ in increasing order of
codimension. 

\subsubsection*{An outer vector field with nondegenerate zeros}
Starting from the lowest dimensional faces of $M$, we construct a smooth vector
field along the boundary faces of $M$ such
that for every boundary point $x$ inside the interior of a face $Y$, $U_x$ lives
in the outer sphere $S^\out_x Y$. 
We extend this vector field smoothly to the interior of $M$, and perturb it to a
vector field $U\in\Omega^0(M,TM)$ transverse to the zero section $M$.
If the perturbation is small enough in $C^0$ norm, the vector field $U$ will
still point in the outer sphere directions at every boundary face.

Define a unit vector field
$V_0:=|U|^{-1} U$ on the complement on the (isolated) zero-set
$Z(U)$ of $U$ in $M$. It is a section of the sphere bundle $\SM\to M$ over the
complement of $Z(U)$. 

\subsubsection*{Blow-up of the singular set of $V_0$}

Let $\tilde{M}$ be the closure of $V_0(M\setminus Z(U))$ in the polyhedral
manifold $\SM$. 

\begin{rem*}
When $M$ is a manifold with corners, the compact polyhedral manifold
$\tilde{M}$ 
is naturally diffeomorphic to
$[M;Z(U)]$, the total space of the blow-up of $Z(U)$ inside $M$. See \cite{Mel}
for the notion of real blow-up of manifolds-with-corners.
\end{rem*}

The boundary of the compactification $\tilde{M}$ in $\SM$ is obtained by gluing
the tangent sphere $S_pM$
near each annulation point $p\in Z(U)$.
More precisely, besides the diffeomorphic image through $V_0$ of the
boundary hyperfaces of $M$, $\partial\tilde{M}$ contains also the
``inner boundary'', i.e., the singular divisor obtained by blowing-up the
annulation points of $U$. Near a non-degenerate annulation point $p\in Z(U)$,
there exist local coordinates in which the vector field $U$ takes the form
\[U(x)=x_1\partial_{x_1}+\ldots
+x_r\partial_{x_r}-x_{r+1}\partial_{x_{r+1}}-\ldots x_{2n}\partial_{x_{2n}}.
\]
where the integer $r$ is the index of $U$ at $p$. The new hypersurface
introduced by blowing up $p$ is just the compact manifold $S_pM$, 
with orientation $(-1)^{r+1}$ times the standard orientation induced from
$T_pM$. We thus separate the boundary of $\tilde{M}$ as
\[\partial\tilde{M}=\left(\bigsqcup_{p\in Z(U)} S_pM\right)\sqcup V_0(\partial
M).
\]
into the union of the inner boundary spheres, and the diffeomorphic image
through $V_0$ of the boundary of $M$.

In order to compute the integral of the Pfaffian on $M$, we will apply Theorem
\ref{thT} to the pull-back bundle $\pi^* TM\to\tilde{M}$ over the compact
polyhedral manifold 
$\tilde{M}$, endowed with the pull-back connection $\pi^* \nabla$. This clever 
construction (due to Chern \cite{Ch1}) is
necessary because $M\setminus Z(U)$ is not compact, so the Stokes formula would
need to take into account the contribution of the singularities $Z(U)$ in the
transgression forms. The r\^ole of the blow-up space $\tilde{M}$ is precisely to
"resolve" this singularity formally.

Since $Z(U)$ is a finite set, it has measure $0$. By naturality, the integral of
the Pfaffian on $M$ can be computed by pull-back on $\tilde{M}$:
\begin{equation}\label{ppt}
\int_M \Pf(R)=\int _{V_0(M\setminus Z(U))} \pi^*\Pf(R)=\int _{\tilde{M}}
\Pf(R^{\pi^*\nabla}).
\end{equation}

Let $\cT^{(1)}(V_0)\in \Omega^{2n-1}(M\setminus Z(U))$ be the first-order
transgression on $M\setminus Z(U)$ from
Definition \ref{deftr} corresponding to the unit vector field $V_0$ interpreted
as a $0$-dimensional simplex
of unit vector fields. Similarly, let $\cT^{(1)}(s)\in \Omega^{2n-1}(SM)$ be the
first-order transgression from
Definition \ref{deftr} corresponding to the canonical unit section $s$ in $\pi^*
TM$ over $\tilde{M}\subset \SM$, interpreted as a $0$-dimensional simplex
of unit vector fields. By naturality, $\pi^* \cT^{(1)}(V_0)= \cT^{(1)}(s)$ on
the complement of the zero set $Z(U)$. 
Now $\partial \tilde{M} = V_0(\partial M) - \sqcup_{p\in Z(U)} S_pM$.
By Stokes formula, 
\begin{align*}
\int _{\tilde{M}} \Pf(R^{\pi^*\nabla})= {}&-\int_{\partial \tilde{M}}
\cT^{(1)}(s)\\
={}& -\int_{\partial {M}} \cT^{(1)}(V_0) +\sum_{p\in Z(U)} \int_{S_pM}
\cT^{(1)}(s).
\end{align*}
\begin{lem}[Chern \cite{Ch1}]
At a annulation point $p\in Z(U)$ of index $r$, the integral on the sphere
$S_pM$ of $\cT^{(1)}(s)$ equals $(-1)^r (2\pi)^n$.
\end{lem}
\begin{proof}
Apply the definition of the transgression in dimension $l=0$ for the canonical
unit vector
field $s$ over $\tilde{M}\subset \SM$. Here the parameter space $X$ is just a
point, hence the terms containing $d^X V$ vanish.
The curvature $R$ vanishes on the vertical sphere $S_p M$ since the connection
is pulled-back from the base, hence the terms with $k<n-1$ also vanish. It
follows that
\[\cT^{(1)}(s)_{|S_pM}=\tfrac{2^{n-1}(n-1)!}{(2n-1)!}\cB\left[
s\left(\pi^*\nabla(s)\right)^{2n-1}\right]
\]
where $\pi^*\nabla(s)$ is given by \eqref{naus}.
\end{proof}
It follows from this lemma and the Poincar\'e-Hopf theorem that the inner
boundary contributions add up, like in the boundary-less case,
to $(2\pi)^n\chi(M)$. We thus rewrite \eqref{ppt} as
\begin{equation}\label{ppt2}
(2\pi)^n\chi(M)=\int_M \Pf(R)+\int_{\partial M}  \cT^{(1)}(V_0).
\end{equation}

This identity finishes the proof of the Gauss-Bonnet theorem 
for closed manifolds. 

If $M$ is a compact manifold with boundary, we can choose 
$V_0$ to be the unit outer normal to $\partial M$. The correction term 
$\cT^{(1)}(\nu_{\partial M})$ is computed in that case as in the final part of
the present proof.

When $M$ has faces of codimension at least $2$, in  Eq.\ \eqref{ppt}
the contribution $\int_Y \cT^{(1)}(V_0)$ of a
boundary hyperface $Y$ depends on the vector field $V_0$, which cannot be chosen
to be the unit normal
to $Y$ simultaneously for all hyperfaces $Y$. In order to write this
contribution in terms of the outer unit normal vector field $\nu_Y$, 
we use the higher transgressions with respect to the families of unit vector
fields
$V_Y, V_{0,Y}$ constructed in Section \ref{oc}.
By Theorem \ref{thT},
\begin{equation}\label{ii}
\int_Y \cT^{(1)}(V_0) = \int_Y \cT^{(1)}(V_Y) - \int_{\partial Y}
\cT^{(2)}(V_{0,Y}).
\end{equation}
The point is, the first term in the right-hand side does not depend on $V_0$, 
while the second is now localized to codimension $2$ faces. The induction
procedure is powered by the next result:
\begin{lem}
Let $Y=F_1\cap\ldots\cap F_l$ be a (possibly disconnected) face of $M$,
where $F_1,\ldots,F_l$ are hyperfaces of $M$. Let $Z_j:=\cap_{i\neq j} F_i$.
Then
\[\sum_{j=1}^l (-1)^j \int_{Y} \cT^{(l)}(V_{0,Z_j})=\int_{Y}\cT^{(l)}(V_Y)
-\int_{\partial Y} \cT^{(l+1)}(V_{0,Y}).
\]
\end{lem}
\begin{proof}
Direct application of Stokes formula and Theorem \ref{thT}.
\end{proof}

Again by induction, for all $d\geq 0$ we have
\begin{align*}
\int_M \Pf(R)={}& (2\pi)^n\chi(M)+\sum_{l=1}^{d}\sum_{Y\in\cF^{(l)}(M)} \int_Y
\cT^{(l)}(V_Y) \\
&+\sum_{Y\in\cF^{(d)}(M)} \int_{\partial
Y}\cT^{(d+1)}(V_{0,Y}).
\end{align*}
The initial step is Eq.\ \eqref{ii}.
Specializing to the maximal codimension $d=2n+1$, we have completely eliminated 
the non-canonical vector field $V_0$ from the formula!
\begin{align*}
\int_M \Pf(R)={}& (2\pi)^n\chi(M)+\sum_{l=1}^{2n}\sum_{Y\in\cF^{(l)}(M)} \int_Y
\cT^{(l)}(V_Y).
\end{align*}
It remains to identify the contribution of each face in terms of intrinsic and
extrinsic 
geometry of the faces (curvature and second fundamental form).
Let $Y$ be a face of codimension $l\geq 1$. from Proposition 
\ref{lemfb}, the map $V_Y:Y^*\times Y\to S^\out
Y$ is a trivialization
of the outer sphere bundle of $Y$, in particular it gives a family indexed by
$Y^*$ 
of unit vector fields in $TM$ along $Y$. Since $Y$ is fixed in this argument, we
write $V=V_Y$ for simplicity.
The transgression $\cT^{(l)}(V_Y)$
is defined by
\[
\cT^{(l)}(V_Y)=\sum_{k=\left\lceil\frac{l}{2}\right\rceil-1}^{n-1}\tfrac{c(n,k,
l-1)}{(l-1)!} \int_{Y^*}\cB 
\left[V (d^{Y^*} V)^{l-1} (\nabla V)^{2k+2-l}R^{n-1-k} \right]\in
\Omega^{2n-l}(Y).
\]
For $x\in Y$, if $\nu_{S^\out Y}$ is the Riemannian volume form on the sphere
$S_x^\out Y$, then $(d^{Y^*} V)^{l-1}$ can be expressed as the pull-back on
$Y^*$ through the map $V$
of the tensor square of the volume form of outer spheres:
\[(d^{Y^*} V)^{l-1}=(l-1)!V^*(\nu_{S_x^\out Y})\otimes \nu_{S_x^\out Y}.\]
Thus the second component of the double form $V(d^{Y^*} V)^{l-1}$ is a multiple
of the volume
form of the normal bundle to $Y$. It follows that
only those terms from $\nabla V$ and $R$ whose second component is a form
tangent to $Y$ may have a nonzero contribution to $\cT^{(l)}(V_Y)$. These terms
are $A^*(V)$, the second fundamental form \eqref{secff} of $Y$ in $M$
interpreted as a
$(1,1)$ form-valued function on $N_Y$, respectively $R_{|Y}$, the components of
the curvature form of $M$ along $Y$. Recall that by the Gauss equation
\[R_{|Y}=R^Y-\tfrac{1}{2} g(A,A)
\]
where the contraction $g(A,A)$ was defined in \eqref{contra}.
For $x\in Y$ we obtain by changing variables in the integral from the polytope
$Y^*$ to the outer sphere
$S_x^\out Y$ using the diffeomorphisms $V:Y^*\to S_x^\out Y$:
\begin{align*}
\cT^{(l)}(V_Y)(x)={}&\sum_{k=\left\lceil\frac{l}{2}\right\rceil-1}^{n-1}c(n,k,
l-1)
\int_{S_x^\out Y}\nu_{S_x^\out Y}\otimes 1 \cdot\\
{}&\cB 
\left[1\otimes \nu_{N_Y} \cdot (-A^*)^{2k+2-l}(R^Y-g(A,A)/2)^{n-1-k}
\right].\qedhere
\end{align*}
\end{proof}

\subsection{Passing from regular to general polyhedral manifolds}\label{s54}
If $M$ is a not regular polyhedral manifold, the above proof breaks down 
because the outer cone bundles are not globally trivial. 
Thus we need a new global argument before applying the local 
computations from the previous sections. Let $I$ denote the unit interval $[0,1]$.

In the pull-back of $TM$ over the polyhedral manifold $\SM\!\times\! I_x\!\times\! I_t$
we consider the pull-back $\cD$ of the connection $D$ from \eqref{defD}
under the projection off the factor $[0,1]_x$ onto $\SM\!\times\! I_t$:
\[\cD=\pi^*\nabla + (1- t) A
\]
where $\pi$ is the projection $\SM\!\times\! I^2\to M$, 
and $A$ is the endomorphism-valued $1$-form defined in \eqref{as} with respect to the tautological section $s$
in $\pi^*TM$,
\[A(W)=s\wedge \pi_*W\in\End^- (\pi^*TM).
\]
Here $t$ is a deformation parameter as before, while $x\in[0,1]$ will be the variable of a conical
deformation of the polyhedral complex $M^\out$ that we now introduce. Recall that we have fixed
a vector field $U$ on $M$ with isolated nondegenerate zeros and outward-pointing along $\partial M$,
and we constructed $V_0=U/\|U\|$ on the complement of $Z(U)$.
In particular, for every $p\in Y\in\cF(M)$, $V_0(p)$ belongs to the convex spherical polytope
$S^\out_p Y$, where $\cF(M)$ is the set of faces of $M$ of codimension at least $1$.

For every face $Y\in\cF(M)$ define a locally trivial bundle of spherical polyhedra
\[\Con_{V_0}(S^\out Y)=\{(\cos x \cdot v_p+\sin x \cdot V_0(p),x); p\in Y, v_p\in S^\out Y,x\in I\}.\]
From Proposition \ref{mopc}, it follows that the set
\[\Con_{V_0}(M^\out)=\bigcup_{Y\in\cF(M)}\Con_{V_0}(S^\out Y)\]
is a polyhedral complex embedded
in $M\!\times\! I$, so $\Con_{V_0}(M^\out)\!\times\! I_t$ is a polyhedral complex embedded
in $M\!\times\! I^2$. We enrich this complex by adding to it certain faces at $x=1$. We start with
the image of $V_0$, i.e., the face $V_0(M\setminus Z(U))\!\times\! \{1\}\!\times\! I_t$ of dimension $2n+1$.
We complete this face with the $2n$-dimensional cylinders $S_pM\!\times\!  \{1\}\!\times\! I_t$ for each annulation point
$p\in Z(U)$. 

\begin{prop}
Let $M$ be a polyhedral manifold of dimension $2n$ with orientable interior. The set 
\[\cP=\Con_{V_0}(M^\out)\times I\bigcup V_0(M\setminus Z(U))\!\times\!\{1\}\!\times\!I \bigcup_{p\in Z(U)}S_pM\!\times\! \{1\}\!\times\!I
\subset \SM\!\times\! I^2\]
is a polyhedral complex of dimension $2n+1$. Its $2n$-boundary is
\begin{align}\nonumber
\partial_{\dim(M)}\cP= {}&{V_0}(M)\!\times\!\{1\}\!\times\! \{1\}-{V_0}(M)\!\times\!\{1\}\!\times\! \{0\} \\
&+\sum_{p\in Z(U)}S_pM\!\times\! \{1\}\!\times\! I- \sum_{Y\in\cF(M)} S^\out Y\!\times\! \{0\}\times I\label{geomb}\\
&+\sum_{Y\in\cF(M)} \Con_{V_0}(S^\out Y) \!\times\! \{1\} - \sum_{Y\in\cF(M)} \Con_{V_0}(S^\out Y)\!\times\! \{0\}.\nonumber
\end{align}
\end{prop}

In order to compute the integral on $M$ of the Pfaffian of $\nabla$, we are going to use the Pfaffian of the connection $\cD$. It is a closed form on $SM\times I^2$, hence by the Stokes formula on polyhedral complexes (Lemma \ref{thSt}), $\int_{\partial \cP}\Pf(R^\cD)=0$.
Moreover, $\Pf(R^\cD)$ vanishes identically on three of the types of faces of $\partial_{2n}\cP$ from \eqref{geomb}: It vanishes on 
$\Con_{V_0}(S^\out Y)\!\times\! \{0\}$ for $Y\in\cF(M)$ and on ${V_0}(M)\!\times\!\{1\}\!\times\! \{0\}$
because at $t=0$ the connection $\cD$ admits a parallel section $s$. It also vanishes
on $\Con_{V_0}(S^\out Y) \!\times\! \{1\}$
for $Y\in\cF(M)$ because along $\{t=1\}$ the connection $\cD$ is the pull-back of $\nabla$ from the base via the projection $\SM\!\times\! I\to M$, 
hence by functoriality the Pfaffian $\Pf(R^\cD)$ is a horizontal form. However, since $\dim(Y)<2n=\rk(TM)$, the Pfaffian of $\nabla$ vanishes on $Y$.

In conclusion of this discussion, by integrating $\Pf(R^\cD)$ on the polyhedral complexes from \eqref{geomb}, we get after using \eqref{ppt2} and 
pull-back by $V_0$:
\[\int_M \Pf(R^\nabla) = (2\pi)^n\chi(M) + \sum_{Y\in\cF(M)} \int_{S^\out Y\times I_t} \Pf(R^\cD).
\]
To conclude the proof, we note that the restriction of $\cD$ to $\{x=0\}$ coincides with the connection $D$ defined in \eqref{defD}. Moreover,
the computation of $\int_{S^\out Y\times I} \Pf(R^D)$, carried out above in the regular case, is local in the base $Y$, so it remains valid even 
without the regularity assumption on $M$.

\subsection{Odd dimensions}

This case follows directly from the even-dimensional case as we now explain.

\begin{theorem}\label{awo}
Let  $(N,g)$ be a compact Riemannian polyhedral manifold of odd dimension
$2n-1$. Then
\begin{align*}
(2\pi)^n\chi(N)=&\sum_{l=1}^{2n-1}\sum_{k=\left\lceil\frac{l-1}{2}\right\rceil}^
{n-1}
\frac{(-1)^{l-1}\pi (2k-1)!!}{(n-1-k)! (2k+1-l)!}\\
{}&\cdot
\sum_{Y\in \cF^{(l)}(N)} \int_{S^\out Y} \cB_Y\left[(R^Y-\tfrac{1}{2}
g(A,A))^{n-1-k} (A^*)^{2k+1-l}\right] |dg|.
\end{align*}
\end{theorem}
By convention, $(-1)!!=0!=0!!=1$, and the $0^{\text{th}}$ power of a double form
is always $1$.
\begin{proof}
Apply theorem \ref{awt} to the product manifold
$M:=N\times I$, where $I$ is the interval $[0,1]$, endowed with the product
metric $h=g+dt^2$. The Euler characteristics of $M$ and $N$ coincide.
We will exploit the fact that 
the vertical vector field $\frac{\partial}{\partial t}$ is parallel along $M$,
but also along $Y\times I$ for every face $Y$ of $N$.

The boundary faces of $M$ fall into two types: 
\begin{itemize}
\item lateral faces of the form
$Y\times I$, and 
\item top or bottom faces of the form $Y\times\{0\}$ or $Y\times
\{1\}$. 
\end{itemize}

The first type of faces do not contribute in the Gauss-Bonnet formula.
Indeed, the curvature form $R^{Y\times I}$ and the second fundamental form 
$A^{Y\times I}$ of ${Y\times I}$ inside $N\times I$ both vanish in the
direction of the parallel vector field $\frac{\partial}{\partial t}$.
It follows that the Berezin integral inside the term from theorem \ref{awt} 
corresponding to the face $Y\times I$ vanishes identically.

The outer spheres in $M$ of the second type of faces, e.g.\ $Y\times\{0\}$, can
be
described as spherical cones
over the outer sphere of $Y$ in $N$.
More precisely, let $V$ be an Euclidean vector space, $S$ the unit sphere in
$V$, $V'$
a hyperplane in $V$ and $A\subset S\cap V'$ a subset of $S$ lying in a subsphere
of codimension $1$.
Let $\{p_0,p_1\}= {V'}^\perp\cap S$, so $p_0$ and $p_1$ are diametrally opposed
and $A$
sits in the equatorial hypersphere orthogonal to $p_0$ and $p_1$.
We define the \emph{spherical cone} of $A$ with respect to $p_0$ as the union of
all geodesic segments in $S$
linking $p_0$ to $A$. The complement of the vertex $p_0$, the \emph{open
spherical cone}, is isometric to a topological product
$A\times [0,\pi/2)$ with the warped product metric $\cos^2(\alpha)
g_A+d\alpha^2$. Note for later use that the volume densities induced by $g$ and
$g_A$ on the fibers of the outer spheres satisfy the identity
\begin{equation}\label{svf}
|dg|= \cos(\alpha)^{\dim(V)-2} d\alpha |dg_A|.
\end{equation}

\begin{lem}
Let $S^\out (Y)$ be the outer sphere of $Y$ in $N$, and $S^\out (Y\times \{0\})$
the outer sphere of $Y\times \{0\}$
in $M$. Then for every $x\in Y$, $S^\out_x (Y\times \{0\})$ is the spherical
cone of $S^\out_x (Y)$ 
with respect to the point $\frac{\partial}{\partial t}$. Moreover, the 
union for all $x\in Y$ of the open
spherical cones form a locally trivial bundle with fiber type $[0,\pi/2)$ over
$S^\out (Y)$
\end{lem}

Using this lemma, we can carry out the integral in $\alpha$ (i.e., along the
fibers 
of the spherical cone fibration) of the integrands from theorem \ref{awt}. For a
fixed $x\in Y$, the curvature  $R^Y$ and the metric contraction of the second 
fundamental form $g(A,A)$ are pull-backs from the base, i.e., they are constant
on the 
outer sphere. The dual $A^*$ of the second fundamental form is linear on the
normal bundle to $Y$
and vanishes at the vertical vector $\partial_t$, hence for $v\in S^\out_xY$,
\[
A^*(\cos(\alpha) v+\sin(\alpha)\partial_t) = \cos(\alpha) A^*(v).
\]
Now the volume form on the spherical cones is given by \eqref{svf}. It follows
that the push-forward
along the fibers of the spherical cones over $S^\out_xY$ of 
$\cB_Y\left[(R^Y-\tfrac{1}{2}
g(A,A))^{n-1-k} (A^*)^{2k+1-l}\right] |dg_M|$
amounts to 
\[
I_{2k}\cB_Y\left[(R^Y-\tfrac{1}{2}
g(A,A))^{n-1-k} (A^*)^{2k+1-l}\right] |dg_N|
\]
where $I_{2k}$ is a scaling factor independent of $x$:
\[
I_{2k}= \int_{-\pi/2}^{\pi/2}\cos^{2k}(\alpha)d\alpha = \frac{\pi
(2k-1)!!}{(2k)!!}.\qedhere
\]
\end{proof}
We have used the convention $0!=0!!=(-1)!!=1$.

\section{Constant-curvature polyhedral manifolds with geodesic faces}

By applying the Gauss-Bonet theorems \ref{awt} and \ref{awo} to the case of 
polyhedral 
manifolds of constant sectional curvature with totally geodesic faces, we 
obtain certain identities for spherical, euclidean and hyperbolic polyhedra in terms of volumes
of faces and measures of outer angles. 

\subsection{Euclidean polyhedra}
Let $M$ be a flat compact polyhedral manifold of dimension $k$ with totally
geodesic faces.
In this case, the Gauss-Bonnet simply states that the sum of the outer angles at
the vertices of $M$
equals the Euler characteristic $\chi(M)$ 
divided by the volume of the $k-1$ sphere.
Indeed, in Theorems \ref{awt} and \ref{awo} the curvature $R^Y$ of the face $Y$ 
and the second fundamental form $A$ of $Y\subset M$ both vanish, so the only
non-zero terms in the 
right-hand side arises for $\codim(Y)=k$, i.e., when $Y$ is a point. In that
case, the integral corresponding to
a vertex $Y$ gives the volume of the outer sphere at the vertex $Y$, and the
formula becomes
\[\vol(S^{k-1})\chi(M) = \sum_{Y\in\cF^{(k)}} \angle^\out Y.
\]
In particular, the Euler characteristic of a flat compact polyhedral manifold 
with totally geodesic faces is always non-negative, and it is necessarily 
positive as soon as $M$ has at least one vertex.
This identity is clear for open polytopes in 
$\bR^{k}$, since the outer spheres of the vertices 
partition the unit sphere $S^{k-1}$ into spherical polytopes with mutually
disjoint interiors. But in general it is not obvious. A direct proof should 
rely on some additivity property of outer angles. 

\subsection{Manifolds of constant sectional curvature with geodesic faces}

Let $(M,g)$ be a compact polyhedral manifold with constant scalar curvature 
$\kk$ and with geodesic faces. Then $R=\frac{\kk}{2}g^2$, valid on every
face. Since the second fundamental form $A$ of any face $Y$ is assumed
to vanish, we also have $A^*=0$. Therefore Theorems \ref{awt} and \ref{awo} give
\begin{align*}
(2\pi)^n\chi(M)= {}&\sum_{k=0}^{n}
\tfrac{2^{k-1} (k-1)!}{(n-k)!}
\sum_{Y\in \cF^{2k}(M)} \int_{S^\out Y} \cB_Y\left[(R^Y)^{n-k} \right] |dg|&
\text{for $\dim(M)=2n$},\\
2^n\pi^{n-1}\chi(N)={}&\sum_{k=1}^{n}
\tfrac{(2k-3)!!}{(n-k)!}
\sum_{Y\in \cF^{2k-1}(N)} \int_{S^\out Y} \cB_Y\left[(R^Y)^{n-k}\right] |dg|&
\text{for $\dim(N)=2n-1$}
\end{align*}
where we recall that $\cF^{(h)}$ denotes the set of faces of codimension $h\geq 0$.
On a face $Y$ of dimension $2j$, we compute moreover for $R=R^Y$
\begin{align*}R^j=\kk^j \frac{(2j)!}{2^j} dg_Y\otimes dg_Y,&& \cB_Y(R^j)=\kk^j
\frac{(2j)!}{2^j} |dg_Y|.
\end{align*}
In conclusion, regardless of the parity of $d=\dim(M)$, the Gauss-Bonnet formula
becomes the sum 
\eqref{gbcc} over the even-dimensional faces of $M^d$ advertised in the
introduction as Theorem \ref{th1}.

Dehn \cite{deh} studied such identities for small dimensions and predicted their
existence in general. 
Allendoerfer-Weil's formula from \cite{AW} was used by Santal\'o \cite{san} for
deducing particular cases of \eqref{gbcc} for polyhedra embedded in a constant
curvature space-form. For a spherical simplex inside $S^k$, the identity was
also announced by Kenzi Sato \cite{sat}.

\subsection{Hyperbolic polyhedra with ideal vertices}
As an extension of the previous example, the hyperbolic identity allows us to
compute the volume of hyperbolic $2n$-polyhedra
with some -- or all -- ideal vertices in terms of outer angles and volumes of
lower-dimensional faces, by passing to the limit the Gauss-Bonnet formula for compact 
polyhedra.
For instance, when $2n=4$ the volume of an ideal hyperbolic $4$-simplex is given
by
\[\vol(M)=-2\pi^2+\frac{\pi}{3}\sum_{Y\in\cF^{(2)}(M)} \angle^\out(Y)\]
where $\angle^\out(Y)$ is the outer dihedral angle of the ideal triangle $Y$ in
$M$, i.e., the angle between the outer normals to the two
hyperfaces containing $Y$. 

\section{Closing remarks}

The Gauss-Bonnet formula \eqref{gbf} 
continues to hold on complete manifolds with warped product
ends with a decay condition on the warping functions \cite{Ro}, and for asymptotically 
cylindrical metrics \cite{Alb}. If $(M,g)$ is a smooth compact Riemannian manifold-with-boundary, 
the Gauss-Bonnet formula contains a correction term along the boundary 
in terms of the second fundamental form \cite{AW,Ch1}. Extensions of this
formula to more general metrics on the interior of a manifold-with-boundary 
were found by Satake \cite{satak} for Riemannian orbifolds, by Albin \cite{Alb}, 
Dai-Wei \cite{DW} and by Cibotaru and the author
\cite{CibM} for manifolds with fibered boundaries, by McMullen \cite{McM}
for cone manifolds, by Anderson \cite{and} for asymptotically hyperbolic Einstein $4$-manifolds,
and again in \cite{CibM} for incomplete edge metrics, to cite only a few results in this direction. 
The proofs typically start from a degeneration process in the Gauss-Bonnet formula for 
manifolds-with-boundary. 

In contrast, the Gauss-Bonnet formula on a Riemannian
polyhedral manifold does not seem to follow from such a degeneration. Although
it may appear tempting to consider a $\epsilon$-neighborhood of $M$ 
as a $C^1$-smoothing of the boundary and then try to compute the limit of the 
boundary integrand as $\epsilon\to 0$ by interpreting the smoothed boundary
as a current like in \cite{Cib}, we were not able to isolate with that approach
the contributions of lower-dimensional faces. 

\subsubsection*{Acknowledgements}
The author was partially supported from the UEFISCDI grant PN-III-P4-ID-PCE-2020-0794 
``Spectral Methods in Hyperbolic Geometry".

\end{document}